\documentclass[11pt]{amsart}
\usepackage{stmaryrd}
\usepackage{amsmath}
\usepackage{amsfonts}
\usepackage{amsthm}
\usepackage{txfonts}
\usepackage{hyperref}
\usepackage{graphicx,amssymb,amsfonts,amsmath,amscd}
\usepackage{    textcomp}
\usepackage[all,ps,cmtip]{xy}
\usepackage{amscd}
\usepackage{xspace}
\usepackage{mathrsfs}
\usepackage{ dsfont }

\newtheorem{thm}{Theorem}[section]

\newtheorem{lemma}[thm]{Lemma}
\newtheorem{cor}[thm]{Corollary}

\newtheorem{def-prop}[thm]{Definition-Proposition}
\newtheorem{prop-def}[thm]{Proposition-Definition}

\theoremstyle{definition} 

\newtheorem{rmk}[thm]{Remark}
\newtheorem{rmks}[thm]{Remarks}
\newtheorem{expl}[thm]{Example}

\newcommand{\C}{{\bf C}}

\newcommand{\Z}{{\bf Z}}
\newcommand{\N}{{\bf N}}

\newcommand{\sL}{{\mathscr L}}

\newcommand{\sO}{{\mathscr O}}

\newcommand{\ie}{\textit {i.e.}~}
\newcommand{\cf}{\textit {cf.}~}

\newcommand{\etc}{\textit {etc.}~}
\newcommand{\resp}{\textit {resp.}~}

\newcommand{\Aut}{\mathop{\rm Aut}\nolimits} 
\newcommand{\dd}{\mathrm{d}} 
\newcommand{\dual}{\mathop{^\vee}\nolimits} 
\newcommand{\Gr}{\mathop{\rm Gr}\nolimits} 

\newcommand{\id}{\mathop{\rm id}\nolimits} 
\renewcommand{\P}{\mathop{\bf P}\nolimits} 
\newcommand{\Pic}{\mathop{\rm Pic}\nolimits} 

\newcommand{\Res}{\mathop{\rm Res}\nolimits} 

\newcommand{\Sym}{\mathop{\rm Sym}\nolimits} 
\newcommand{\vide}{{\varnothing}} 

\newcommand{\PGL}{\mathop{\rm PGL}\nolimits}

\renewcommand{\bar}{\overline}

\newcommand{\lra}{\xrightarrow}

\newcommand{\cart}{\ar@{}[dr]|\square} 
\renewcommand{\dual}{^{\vee}} 
\newcommand{\isom}{\simeq} 

\newcommand{\diag }{\mathop{\rm {diag}}\nolimits}
\newcommand{\Span }{\mathop{\rm {Span}}\nolimits}

\begin{document}

\title{Classification of polarized symplectic automorphisms of Fano varieties of cubic fourfolds}
\author{Lie Fu}
\address{D\'epartement de Math\'ematiques et Applications, \'Ecole Normale Sup\'erieure, 45 Rue d'Ulm, 75230 Paris Cedex 05, France}
\email{lie.fu@ens.fr}
\begin{abstract}
We classify the polarized symplectic automorphisms of Fano varieties
of smooth cubic fourfolds (equipped with the Pl\"ucker
polarization) and study the fixed loci.
\end{abstract}
\maketitle

\setcounter{section}{-1}
\section{Introduction}\label{sect:intro}

The purpose of this paper is to classify the polarized symplectic automorphisms of the irreducible holomorphic symplectic projective varieties constructed by Beauville and Donagi \cite{MR818549}, namely, the Fano varieties of (smooth) cubic fourfolds.

Finite order symplectic automorphisms of K3 surfaces have been studied in detail by Nikulin in \cite{MR544937}. A natural generalization of K3 surfaces to higher dimensions is the notion of \emph{irreducible holomorphic symplectic manifolds} or \emph{hyper-K\"ahler manifolds} (\cf \cite{MR730926}), which by definition is a simply connected compact K\"ahler manifold with $H^{2,0}$ generated by a symplectic form (\ie nowhere degenerate holomorphic 2-form). Initiated by Beauville  \cite{MR728605}, some results have been obtained in the study of automorphisms of such manifolds. Let us mention \cite{MR2805992}, \cite{MR2957195}, \cite{MR2932167}, \cite{MR2967237}.

In \cite{MR818549}, Beauville and Donagi show that the Fano varieties of lines of smooth cubic fourfolds provide an example of a 20-dimensional family of irreducible holomorphic symplectic projective fourfolds. We propose to classify the polarized \emph{symplectic} automorphisms of this family.

Our result of classification is shown in the table
below\footnote{Please see the next page.}. We firstly make
several remarks concerning this table:
\begin{itemize}
\item As is remarked in \S\ref{sect:natural}, such an automorphism comes from a (finite order) automorphism of the cubic fourfold itself. Hence we express the automorphism in the fourth column as an element $f$ in $\PGL_6$.
\item In the third column, $n$ is the order of $f$, which is \emph{primary} (\ie a power of a prime number). The reason why we only listed the automorphisms with primary order is that every finite order automorphism is a product of commuting automorphisms with primary orders, by the structure of cyclic groups. See Remark \ref{rmk:final}.
\item We give an explicit basis of the family in the fifth column.
\item In the last column, we work out the fixed loci for a generic member. For geometric descriptions of the fixed loci, see
\S\ref{sect:fixed}.
\item   The Family I in our classification has been discovered in \cite{Mongardi}. 
\item   The Family V-(1) in our classification has been studied in
\cite{MR2967237}, where the fixed locus and the number of moduli are
calculated.
\item  The classification of prime order automorphisms of cubic fourfolds has been done in \cite{MR2820585}. I am also informed by G. Mongardi that he classifies the prime order symplectic automorphisms of hyper-K\"ahler varieties which are of $K3^{[n]}$-deformation type in his upcoming thesis.
\end{itemize}

\newpage
\begin{thm}[Classification]\label{thm:classification}
Here is the list of all families of cubic fourfolds equipped with an automorphism of primary order whose general member is smooth, such that the induced actions on the Fano varieties of lines are symplectic.\\
\begin{tabular}{|c|c|c|c|c|c|}
 \hline
Family &$p$ & $n=p^m$ & automorphism & basis for $\bar B$  & fixed loci\\
\hline
0& 1 & 1 & $f=\id_{\P^5}$ &  degree 3 monomials  & $F(X)$\\
\hline
I&11 & 11 & $f=\diag(1, \zeta, \zeta^{-1}, \zeta^3, \zeta^{-5}, \zeta^4)$& $x_0^2x_1$& 5 points\\
&&&$\zeta=e^{\frac{r}{11}\cdot 2\pi\sqrt{-1}}, 1\leq r\leq 10$ &$x_1^2x_2$&\\
&&&& $x_2^2x_3$&\\
&&&& $x_3^2x_4$&\\
&&&&$x_4^2x_0$&\\
&&&& $x_5^3$&\\
\hline
II&7 & 7 & $f=\diag(1, \zeta, \zeta^{-1}, \zeta^3, \zeta^{2}, \zeta^{-3})$& $x_0^2x_1$& 9 points\\
&&&$\zeta=e^{\frac{r}{7}\cdot 2\pi\sqrt{-1}}, 1\leq r\leq 6$ & $x_1^2x_2$&\\
&&&&$x_2^2x_3$&\\
&&&&$x_3^2x_4$&\\
&&&&$x_4^2x_5$&\\
&&&& $x_5^2x_0$&\\
&&&&$x_0x_2x_4$&\\
&&&&$x_1x_3x_5$&\\
\hline
III &5 & 5 & $f=\diag(1, \zeta, \zeta^{-1}, \zeta^{-2}, \zeta^2, \zeta^{2})$& $x_0^2x_1$&14 points\\
&&&$\zeta=e^{\frac{r}{5}\cdot 2\pi\sqrt{-1}}, 1\leq r\leq 4$ & $x_1^2x_2$&\\
&&&&$x_2^2x_3$&\\
&&&&$x_3^2x_0$&\\
&&&&$x_4^2x_5$&\\
&&&&$x_5^2x_4$&\\
&&&&$x_5^3$&\\
&&&&$x_4^3$&\\
&&&&$x_0x_2x_4$&\\
&&&&$x_0x_2x_5$&\\
&&&&$x_1x_3x_4$&\\
&&&&$x_1x_3x_5$&\\
\hline
\end{tabular}

\noindent\begin{tabular}{|c|c|c|c|c|c|}
 \hline
IV-(1)&3 & 3 & $f=\diag(1,1,1,1,\omega,\omega^2)$& degree 3 monomials on $x_0,\cdots, x_3$&27 points\\
&&&$\omega=e^{\frac{2\pi\sqrt{-1}}{3}}$ & $x_4^3$&\\
&&&& $x_5^3$&\\
&&&& $x_0x_4x_5$&\\
&&&& $x_1x_4x_5$&\\
&&&&$x_2x_4x_5$&\\
&&&&$x_3x_4x_5$&\\
\hline
IV-(2)&3 & 3 & $f=\diag(1,1,1,\omega,\omega,\omega)$& degree 3 monomials on $x_0,x_1,x_2$& an abelian surface\\
&&&$\omega=e^{\frac{2\pi\sqrt{-1}}{3}}$ &  degree 3 monomials on $x_3,x_4,x_5$& \\
\hline
IV-(3)&3 & 3 & $f=\diag(1,1,\omega,\omega,\omega^2,\omega^2)$& degree 3 monomials on $x_0,x_1$& 27 points\\
&&&$\omega=e^{\frac{2\pi\sqrt{-1}}{3}}$ & degree 3 monomials on $x_2,x_3$&\\
&&&& degree 3 monomials on $x_4,x_5$&\\
&&&& $x_0x_2x_4$&\\
&&&& $x_0x_2x_5$&\\
&&&& $x_0x_3x_4$&\\
&&&& $x_0x_3x_5$&\\
&&&& $x_1x_2x_4$&\\
&&&& $x_1x_2x_5$&\\
&&&& $x_1x_3x_4$&\\
&&&& $x_1x_3x_5$&\\
\hline
IV-(4)&3 & 9 & $f=\diag(1, \zeta^{-3}, \zeta^{3}, \zeta, \zeta^4, \zeta^{-2})$& $x_0^2x_1$& 9 points\\
&&&$\zeta=e^{\frac{r}{9}\cdot 2\pi\sqrt{-1}}, r=1,2,4,5,7,8$ &$x_1^2x_2$&\\
&&&& $x_2^2x_0$&\\
&&&& $x_3^2x_4$&\\
&&&& $x_4^2x_5$&\\
&&&& $x_5^2x_3$&\\
\hline
IV-(5)&3 & 9 & $f=\diag(1, \zeta^{3}, \zeta^{-3}, \zeta, \zeta, \zeta^{4})$& $x_0^2x_1$&9 points\\
&&&$\zeta=e^{\frac{r}{9}\cdot 2\pi\sqrt{-1}}, r=1,2,4,5,7,8$ & $x_1^2x_2$&\\
&&&&$x_2^2x_0$&\\
&&&&$x_3^2x_4$&\\
&&&&$x_3x_4^2$&\\
&&&&$x_3^3$&\\
&&&&$x_4^3$&\\
&&&&$x_5^3$&\\
\hline
\end{tabular}

\noindent\begin{tabular}{|c|c|c|c|c|c|}
 \hline
V-(1)&2 & 2 & $f=\diag(1,1,1,1,-1,-1)$&  degree 3 monomials on $x_0,\cdots, x_3$& 28 points and\\
&&&& $x_0x_5^2, x_1x_5^2, x_2x_5^2, x_3x_5^2$& a K3 surface\\
&&&& $x_0x_4^2, x_1x_4^2, x_2x_4^2, x_3x_4^2$&\\
&&&&$x_0x_4x_5, x_1x_4x_5, x_2x_4x_5, x_3x_4x_5$&\\
\hline
V-(2)(a)&2 & 4 & $f=\diag(1,1,-1,-1,\sqrt{-1},-\sqrt{-1})$& $x_0^3,x_0^2x_1, x_0x_1^2,x_1^3$&15 points\\
&&&& $x_0x_2^2$&\\
&&&& $x_1x_2^2$ &\\
&&&& $x_0x_3^2$ &\\
&&&& $x_1x_3^2$ &\\
&&&& $x_0x_2x_3$ &\\
&&&& $x_1x_2x_3$ &\\
&&&& $x_2x_4^2$ &\\
&&&& $x_3x_4^2$ &\\
&&&& $x_2x_5^2$ &\\
&&&& $x_3x_5^2$ &\\
&&&& $x_0x_4x_5$& \\
&&&& $x_1x_4x_5$& \\
\hline
V-(2)(b)&2 & 4 & $f=\diag(1,1,-1,-1,\sqrt{-1},-\sqrt{-1})$& $x_2\cdot$ degree 2 monomials on $x_0, x_1$& 15 points\\
&&&& $x_3\cdot$ degree 2 monomials on $x_0, x_1$&\\
&&&& $x_2^3$ &\\
&&&& $x_3^3$ &\\
&&&& $x_2^2x_3$ &\\
&&&& $x_2x_3^2$ &\\
&&&& $x_0x_4^2$ &\\
&&&& $x_1x_4^2$ &\\
&&&& $x_0x_5^2$ &\\
&&&& $x_1x_5^2$ &\\
&&&& $x_2x_4x_5$ &\\
&&&& $x_3x_4x_5$ &\\
\hline
V-(3)&2 & 8 & $f=\diag(1,-1,\zeta^{2},\zeta^{-2},\zeta,\zeta^3)$& $x_0^3$& 6 points\\
&&&$\zeta=e^{\frac{r}{8}\cdot2\pi\sqrt{-1}}$, $r=\pm1\mod 8$& $x_0x_1^2$ &\\
&&&& $x_1x_2^2$ &\\
&&&& $x_1x_3^2$ &\\
&&&& $x_0x_2x_3$& \\
&&&& $x_3x_4^2$ &\\
&&&& $x_2x_5^2$ &\\
&&&& $x_1x_4x_5$& \\
\hline
\end{tabular}

\end{thm}

The structure of this paper is as follows. In \S\ref{sect:natural}
we set up the basic notation, and show that any polarized
automorphism of the Fano variety comes from a finite order
automorphism of the cubic fourfold. Then in \S\ref{sect:setting} we
reinterpret the assumption of being symplectic into a numerical
equation by using Griffiths' theory of residue. Finally we do the
classification in \S\ref{sect:classificatoin}. The basic observation
is that the generic smoothness of the family of cubics imposes
strong combinatoric constrains. 

Throughout this paper, we work over the field of complex numbers with a fixed choice of $\sqrt{-1}$.

\noindent\textbf{Acknowledgements:} I would like to express my gratitude to Olivier Benoist, Zhi Jiang and Junyi Xie for their careful reading of the preliminary version of the paper as well as many helpful suggestions. I also want to thank Chiara Camere for a nice conversation at Luminy and for pointing out to me the paper \cite{MR2820585}. Finally I wish to thank Giovanni Mongardi for informing me about his related classification work in his upcoming thesis.

\section{Fano varieties of lines of cubic fourfolds}\label{sect:natural}
First of all, let us fix the notation and make some basic
constructions. Let $V$ be a 6-dimensional $\C$-vector space, and
$\P^5:=\P(V)$ be the corresponding projective space of 1-dimensional
subspaces of $V$. Let $X\subset \P^5$ be a smooth cubic fourfold. The following
 subvariety of the Grassmannian $\Gr(\P^1, \P^5)$
\begin{equation}\label{eqn:F}
F(X):=\left\{[L]\in \Gr(\P^1, \P^5)~|~ L\subset X\right\}
\end{equation}
is called the \emph{Fano variety of lines}\footnote{In the
scheme-theoretic language, $F(X)$ is defined to be the zero locus of
$s_T\in H^0\left(\Gr(\P^1, \P^5), \Sym^3S\dual\right)$, where $S$ is
the universal tautological subbundle on the Grassmannian, and $s_T$
is the section induced by $T$ using the morphism of vector bundles
$\Sym^3H^0(\P^5,\sO(1))\otimes \sO\to \Sym^3 S\dual$ on $\Gr(\P^1,
\P^5)$.} of $X$. It is well-known that $F(X)$ is a 4-dimensional
smooth projective variety. Throughout this paper, we always equip
$F(X)$ with the polarization $\sL$, which is by definition the
restriction to it of the Pl\"ucker line bundle on the ambient
Grassmannian $\Gr(\P^1, \P^5)$.

Consider the incidence variety (\ie the universal projective line
over $F(X)$):
$$P(X):=\left\{(x, [L])\in X\times F(X)~|~ x\in L\right\},$$
and then the following natural correspondence:
\begin{displaymath}
  \xymatrix{
  P(X)\ar[r]^{q} \ar[d]_{p} & X\\
  F(X) & \\
  }
\end{displaymath}
we have the following
\begin{thm}[Beauville-Donagi \cite{MR818549}]\label{thm:BD}
 Keeping the above notation,
\begin{itemize}
  \item[$(i)$] $F(X)$ is a 4-dimensional irreducible holomorphic
  symplectic projective variety, \ie $F(X)$ is simply-connected and
  $H^{2,0}(F(X))=\C\cdot\omega$ with $\omega$ a no-where degenerate
  holomorphic 2-form.
  \item[$(ii)$] The correspondence $$p_*q^*: H^4(X,\Z)\to H^2(F(X), \Z)$$ is an
isomorphism of Hodge structures.
\end{itemize}
\end{thm}

By definition, an automorphism $\psi$ of $F(X)$ is called
\emph{polarized}, if it preserves the Pl\"ucker polarization:
$\psi^*\sL\isom\sL$. Now we investigate the meaning for an
automorphism of $F(X)$ to be polarized.

\begin{lemma}\label{lemma:natural}
An automorphism $\psi$ of $F(X)$ is polarized if and only if it is
induced from an automorphism of the cubic fourfold $X$ itself.
\end{lemma}
\begin{proof}
See \cite[Proposition 4]{RmkTorelli}.
%
\end{proof}

Define $\Aut(X)$ to be the automorphism group of $X$, and
$\Aut^{pol}(F(X), \sL)$ or simply $\Aut^{pol}(F(X))$ to be the group
of polarized automorphisms of $F(X)$. Then Lemma \ref{lemma:natural}
says that the image of the natural homomorphism $\Aut(X)\to
\Aut(F(X))$ is exactly $\Aut^{pol}(F(X))$. This homomorphism of
groups is clearly injective (since for each point of $X$ there
passes a 1-dimensional family of lines), hence we have
\begin{cor}\label{cor:induced}
  The natural morphism $$\Aut(X)\lra{\isom} \Aut^{pol}(F(X))$$ which sends an automorphism $f$ of $X$ to the
  induced (polarized) automorphism $\hat f$ of $F(X)$ is an isomorphism.
\end{cor}

\begin{rmk}\label{rmk:finite}
This group is a \emph{finite} group. Indeed, since $\Pic(X)=\Z\cdot\sO_X(1)$, all its automorphisms come from linear automorphisms of $\P^5$, hence $\Aut(X)$ is
a closed subgroup of $\PGL_6$ thus of finite type. On the other hand, $H^0(F(X),T_{F(X)})=H^{1,0}(F(X))=0$, which implies that the group considered is also discrete, therefore finite.
\end{rmk}

By Corollary \ref{cor:induced}, the classification of polarized symplectic automorphisms of $F(X)$ is equivalent to the classification of automorphism of cubic fourfolds such that the induced action satisfies the symplectic condition. The first thing to do is to find a  reformulation of this \emph{symplectic condition} purely in terms of the action on the cubic fourfold:

\section{The symplectic condition}\label{sect:setting}
The content of this section has been done in my paper \cite[Section 1]{SympAut}. For the sake of completeness, we briefly reproduce it here.
Keeping the notation in the previous section.
Suppose the cubic fourfold $X\subset \P^5$ is defined by a polynomial $T\in H^0(\P^5, \sO(3))=\Sym^3V\dual$. Let
$f$ be an automorphism of $X$. By Remark \ref{rmk:finite}, $f$ is the restriction of a finite order linear automorphism of
$\P^5$ preserving $X$, still denoted by $f$. Let $n\in \N^+$ be its order. We can assume without loss of generality
that $f:\P^5\to \P^5$ is given by:
\begin{equation}\label{eqn:f}
  f:[x_0: x_1:\cdots: x_5] \mapsto [\zeta^{e_0}x_0:
  \zeta^{e_1}x_1:\cdots:
\zeta^{e_5}x_5],
\end{equation}
where $\zeta=e^{\frac{2\pi \sqrt{-1}}{n}}$ is a primitive $n$-th
root of unity and $e_i\in \Z/n\Z$ for $i=0,\cdots, 5$.

It is clear that $X$ is preserved by $f$if and only if the defining equation $T$ is contained in an eigenspace of $\Sym^3V\dual$. More precisely: let the coordinates $x_0,
x_1, \cdots, x_5$ of $\P^5$ be a basis of $V\dual$, then $\left\{\underline
x^{~\underline \alpha}\right\}_{\underline \alpha\in \Lambda}$ is a
basis of $\Sym^3V\dual=H^0(\P^5, \sO(3))$, where $\underline
x^{~\underline \alpha}$ denotes $x_0^{\alpha_0}x_1^{\alpha_1}\cdots
x_5^{\alpha_5}$. Define
\begin{equation}\label{eqn:Lambda}
\Lambda:=\left\{\underline \alpha=(\alpha_0, \cdots, \alpha_5)\in
\N^5~|~ \alpha_0+\cdots+\alpha_5=3\right\}.
\end{equation}
We write the eigenspace decomposition of $\Sym^3V\dual$:
$$\Sym^3V\dual=\bigoplus_{j\in \Z/n\Z}\left(\bigoplus_{\underline\alpha\in \Lambda_j}\C\cdot\underline x^{~\underline
\alpha}\right),$$ where for each $j\in \Z/n\Z$, we define the subset
of $\Lambda$
\begin{equation}\label{eqn:Lambdaj}
\Lambda_j:=\left\{\underline \alpha=(\alpha_0, \cdots, \alpha_5)\in
\N^5~|~ \substack{\alpha_0+\cdots+\alpha_5=3\\
e_0\alpha_0+\cdots +e_5\alpha_5=j \mod n }\right\}.
\end{equation}
and the eigenvalue of $\bigoplus_{\underline\alpha\in
\Lambda_j}\C\cdot\underline x^{~\underline \alpha}$ is thus $\zeta^j$.
Therefore, explicitly speaking, we have:
\begin{lemma}\label{lemma:preserve}
A cubic fourfold $X$ is preserved by the $f$ in (\ref{eqn:f}) if
and only if there exists a $j\in \Z/n\Z$ such that its defining
polynomial $T\in \bigoplus_{\underline\alpha\in
\Lambda_j}\C\cdot\underline x^{~\underline \alpha}$.
\end{lemma}

Then we deal with the symplectic condition. Note that Theorem \ref{thm:BD} (ii) says in particular that  $$p_*q^*:
H^{3,1}(X)\lra{\isom}H^{2,0}(F(X))$$ is an isomorphism. If $X$ is
equipped with an action $f$ as before, we denote by $\hat f$ the
induced automorphism of $F(X)$. Since the construction of $F(X)$ as well as the correspondence $p_*q^*$ are both
functorial with respect to $X$, the condition that $\hat f$ is
\emph{symplectic} \ie $\hat f^*$ acts on $H^{2,0}(F(X))$ as identity, is equivalent to the condition that
$f^*$ acts as identity on $H^{3,1}(X)$. Work it out explicitly, we
arrive at the congruence equation (\ref{eqn:j}) in the following

\begin{lemma}[Symplectic condition]\label{lemma:symp}
Let $f$ be the linear automorphism in (\ref{eqn:f}), and $X$ be a cubic fourfold defined by equation $T$. Then the followings are equivalent:
\begin{itemize}
 \item $f$ preserves $X$ and the
induced action $\hat f$ on $F(X)$ is symplectic;
 \item  There
exists a $j\in \Z/n\Z$ satisfying the equation
\begin{equation}\label{eqn:j}
e_0+e_1+\cdots+e_5=2j \mod n,
\end{equation}
such that the defining polynomial $T\in
\bigoplus_{\underline\alpha\in \Lambda_j}\C\cdot\underline
x^{~\underline \alpha}$, where as in (\ref{eqn:Lambdaj})
\begin{equation*}
  \Lambda_j:=\left\{\underline \alpha=(\alpha_0, \cdots, \alpha_5)\in
\N^5~|~ \substack{\alpha_0+\cdots+\alpha_5=3\\
e_0\alpha_0+\cdots +e_5\alpha_5=j \mod n }\right\}.
\end{equation*}
\end{itemize}
\end{lemma}
\begin{proof}
  Firstly, the condition that $f$ preserves $X$ is given in Lemma
  \ref{lemma:preserve}. As is remarked before the lemma, $\hat f$ is symplectic if and
  only if $f^*$ acts as identity on $H^{3,1}(X)$. On the other hand, by
  Griffiths' theory of the Hodge structures of hypersurfaces (\cf \cite[Chapter 18]{MR1997577}), $H^{3,1}(X)$ is
  generated by the residue $\Res{\frac{\Omega}{T^2}}$, where $\Omega:=\sum_{i=0}^5(-1)^ix_i\dd x_0\wedge\cdots\wedge\widehat{\dd
  x_i}\wedge\cdots\wedge \dd x_5$ is a generator of $H^0(\P^5,
  K_{\P^5}(6))$. $f$ being defined in (\ref{eqn:f}), we find
  $f^*\Omega=\zeta^{e_0+\cdots+e_5}\Omega$ and $f^*(T)=\zeta^jT$.
  Hence the action of $f^*$ on $H^{3,1}(X)$ is the multiplication by
  $\zeta^{e_0+\cdots+e_5-2j}$, from where we obtain the equation (\ref{eqn:j}).
\end{proof}

\section{Classification}\label{sect:classificatoin}
We now turn to the classification. Retaining the notation of \S\ref{sect:setting}: (\ref{eqn:f}),(\ref{eqn:Lambda}),(\ref{eqn:Lambdaj}),(\ref{eqn:j}). We define the parameter space
\begin{equation}\label{eqn:Bbar}
\bar B:= \P\left(\bigoplus_{\underline \alpha\in \Lambda_j}\C\cdot
\underline x^{~\underline \alpha}\right).
\end{equation}
Let $B\subset \bar B$ be the open subset parameterizing the smooth ones.

In this paper we are only interested in the \emph{smooth} cubic fourfolds, that is the case when $B\neq \vide$, or equivalently, when a general member of $\bar B$ is smooth. The easy observation below (see Lemma \ref{lemma:2}) which makes the classification feasible is that this non-emptiness condition imposes strong combinatoric constrains on the defining equations.

\begin{lemma}\label{lemma:2}
If a general member in $\bar B$ is smooth then for each
$i\in\{0,1,\cdots, 5\}$, there exists $i'\in\{0,1,\cdots, 5\}$, such
that ${x_i}^2x_{i'}\in \bar B$.
\end{lemma}
\begin{proof}
 Suppose on the contrary that, without loss of generality, for $i=0$, none of the monomials $x_0^3$, $x_0^2x_1$, $x_0^2x_2$, $x_0^2x_3$, $x_0^2x_4$, $x_0^2x_5$ are contained in $\bar B$, then every equation in this family can be written in the following form:
 $$x_0Q(x_1, \cdots, x_5)+C(x_1, \cdots, x_5),$$
where $Q$ (\resp $C$) is a homogeneous polynomial of degree 2 (\resp 3). It is clear that $[1,0,0,0,0,0]$ is always a singular point, which is a contradiction.
\end{proof}

Since a finite-order automorphism amounts to the action of a finite
cyclic group, which is the product of some finite cyclic groups with
order equals to a power of a prime number, we only have to classify
automorphisms of \emph{primary} order, that is $n=p^m$ for $p$ a prime number and $m\in\N_+$. To get general results for any order from the classification of primary order case, see Remark \ref{rmk:final}. We thus
assume $n=p^m$ in the sequel.

For the convenience of the reader, we summarize all the relevant
equations:
\begin{equation}\label{eqn:forclassify}
\begin{cases}e_0+e_1+\cdots+e_5=2j \mod p^m;\\
\alpha_0+\cdots+\alpha_5=3;~~~ \alpha_i\in \N;\\
e_0\alpha_0+\cdots +e_5\alpha_5=j \mod p^m;\\
(*)~~\forall i, \exists i' \text{ such that } 2e_i+e_{i'}=j \mod p^m
\end{cases}
\end{equation}
where the last condition $(*)$ comes from Lemma \ref{lemma:2}.

We associate to each solution of $(\ref{eqn:forclassify})$ a
diagram, \ie a finite oriented graph, as follows:
\begin{itemize}
 \item[$(i)$] The vertex set is the quotient set of $\{0, \cdots, 5\}$ with respect to the equivalence relation defined by: $i_1\sim i_2$ if and only if $e_{i_1}=e_{i_2} \mod p^m$.
 \item[$(ii)$] For each pair $(i, i')$ satisfying $2e_i+e_{i'}=j \mod p^m$, there is an arrow from $i$ to $i'$.
\end{itemize}

Clearly the arrows in $(ii)$ are well-defined cause we have taken
into account of the equivalence relation in $(i)$. It is also
obvious that each vertex can have at most one arrow going out.
Thanks to the condition $(*)$ in (\ref{eqn:forclassify}), we know
that each vertex has exactly one arrow going out.

If $p\neq 2$, it is easy to see that each vertex has at most one
arrow coming in. Since the total going-out degree should coincide
with the total coming-in degree, each vertex has exactly one arrow
coming-in. As a result, the diagram is in fact a disjoint union of
several cycles\footnote{Here we use the terminology `cycle' in the sens of graph theory: it means a loop in a oriented graph with no arrow repeated. The \emph{length} of a cycle will refer to the number of arrows appearing in it.} in this case.

Before the detailed case-by-case analysis, let us point out that a
cycle in the diagram would have some congruence implications:
\begin{lemma}\label{lemma:cycle}
\begin{itemize}
  \item[$(i)$] There cannot be cycles of length 2.
  \item[$(ii)$] If $p\neq 3$, there are at most one cycle of length 1.
  \item[$(iii)$] If there is a cycle of length $l=$3, 4, 5 or 6, then
  $p$ divides $\frac{(-2)^l-1}{3}$.
\end{itemize}
\end{lemma}
\begin{proof}
$(i)$ It is because $2e_i+e_{i'}=2e_{i'}+e_i \mod p^m$ will imply
$e_i=e_i' \mod p^m$, contradicts to the definition of a cycle.\\
$(ii)$ A cycle of length 1 means $3e_i=j \mod p^m$, and when $p\neq
3$, $e_i$ is determined by $j$.\\
$(iii)$ Without loss of generality, we can assume that the cycle is
given by:
$$2e_0+e_1=2e_1+e_2=\cdots=2e_{l-2}+e_{l-1}=2e_{l-1}+e_0=j \mod p^m.$$
This system of congruence equations implies that
\begin{equation}\label{eqn:cycle}
 \left((-2)^l-1\right)e_0=\frac{(-2)^l-1}{3}\cdot j \mod p^m.
\end{equation}
If $p$ does not divide $\frac{(-2)^l-1}{3}$, then by
(\ref{eqn:cycle}), we have $j=3e_0 \mod p^m$, and hence
$e_0=e_1=\cdots=e_{l-1}$, contradicting to the definition of a
cycle.
\end{proof}

Then we work out the classification case by case. The result is
summarized in Theorem \ref{thm:classification}.

\noindent\textbf{Case 0.} When $p\neq 2, 3, 5, 7 \text{ or } 11$.\\
If we have a cycle of length $l\geq 3$, since in Lemma \ref{lemma:cycle}, $\frac{(-2)^l-1}{3}$ could only be $ -3, 5, -11, 21$, all of which are prime to $p$, this will lead to a contradiction.
Therefore we only have cycles of length 1. As $p\neq 3$, $3^{-1}j \mod p^m$ is well-defined, hence we have $e_0=e_1=\cdots=e_5$. As a result, $f$ is the identity action of $\P^5$, which is Family 0 in Theorem \ref{thm:classification}. \\

\noindent\textbf{Case I.} When $p=11$.\\ As in the previous case, by
Lemma \ref{lemma:cycle}, cycles of length 2,3,4 or 6 cannot occur.
Thus the only possible lengths of cycles are 1 and 5. If there is no
cycle of length 5, then as before, since $p\neq 3$, all $e_i$'s will
be equal and $f$ will be the identity. Let the 5-cycle be
$$2e_0+e_1=2e_1+e_2=2e_2+e_3=2e_3+e_4=2e_4+e_0=j \mod 11^m.$$ As in (\ref{eqn:cycle}),
$$33e_0=11j \mod 11^m.$$ There thus exists $r\in \Z/11\Z$, such that $j=3e_0+r\cdot11^{m-1} \mod 11^m$, and
\begin{equation*}
 \begin{cases}
e_0=e_0;\\
e_1=e_0+r\cdot11^{m-1};\\
e_2=e_0-r\cdot11^{m-1};\\
e_3=e_0+3r\cdot11^{m-1};\\
e_4=e_0-5r\cdot11^{m-1};\\
e_5=e_0+4r\cdot11^{m-1};
 \end{cases}
\mod 11^m
\end{equation*}
where the last equality comes from the first equation in
(\ref{eqn:forclassify}). Clearly, $r$ cannot be 0, otherwise
$f=\id_{\P^5}$. One verifies easily that it is indeed a solution:
$3e_5=j$. As a result, remembering that we are in the projective
space $\P^5$,
$$f=\left(\begin{array}{cccccc}
1&&&&&\\
           &\zeta &&&&\\
&&\zeta^{-1}&&&\\
&&&\zeta^{3}&&\\
&&&&\zeta^{-5}&\\
&&&&&\zeta^{4}          \end{array} \right)$$ where
$\zeta=e^{\frac{r}{11}\cdot 2\pi\sqrt{-1}}$ and $1\leq r\leq 10$.
That is, $p=11, m=1, (e_0, \cdots, e_5)= (0,1,-1,3,-5,4)$. Going
back to (\ref{eqn:forclassify}), we easily work out all solutions
for $\alpha_i$'s:
\begin{eqnarray*}
 (\alpha_0, \cdots, \alpha_5)&=&(2,1,0,0,0,0), (0,2,1,0,0,0), (0,0,2,1,0,0), (0,0,0,2,1,0),\\&& (1,0,0,0,2,0), (0,0,0,0,0,3).
\end{eqnarray*}
Thus the corresponding family
$$\bar B=\P\left(\Span\langle x_0^2x_1, x_1^2x_2, x_2^2x_3, x_3^2x_4, x_4^2x_0, x_5^3 \rangle\right).$$
In order to verify the smoothness of a general member, it suffices
to give one smooth cubic fourfold in $\bar B$. For example,
$x_0^2x_1+x_1^2x_2+x_2^2x_3+x_3^2x_4+x_4^2x_0+x_5^3$ is smooth.
This is Family I in Theorem \ref{thm:classification}. We would like to mention that this example has been discovered in \cite{Mongardi}.\\

\noindent\textbf{Case II.} When $p=7$.\\ As before, by Lemma
\ref{lemma:cycle}, cycles of length 2,3,4 or 5 cannot occur. Thus
the only possible lengths of cycles are 1 and 6; and except the
trivial Family 0, there must be a 6-cycle:
$$2e_0+e_1=2e_1+e_2=2e_2+e_3=2e_3+e_4=2e_4+e_5=2e_5+e_0=j \mod 7^m.$$ As in (\ref{eqn:cycle}),
$$63e_0=21j \mod 7^m.$$ There thus exists $r\in \Z/7\Z$, such that $j=3e_0+r\cdot7^{m-1} \mod 7^m$, and
\begin{equation*}
 \begin{cases}
e_0=e_0;\\
e_1=e_0+r\cdot7^{m-1};\\
e_2=e_0-r\cdot7^{m-1};\\
e_3=e_0+3r\cdot7^{m-1};\\
e_4=e_0+2r\cdot7^{m-1};\\
e_5=e_0-3r\cdot7^{m-1};
 \end{cases}
\mod 7^m
\end{equation*}
Clearly, $r$ cannot be 0, otherwise $f=\id_{\P^5}$.
 One verifies easily that it is indeed a solution: $e_0+\cdots+e_5=2j$. As a result,
$$f=\left(\begin{array}{cccccc}
1&&&&&\\
           &\zeta &&&&\\
&&\zeta^{-1}&&&\\
&&&\zeta^{3}&&\\
&&&&\zeta^{2}&\\
&&&&&\zeta^{-3}          \end{array} \right)$$ where
$\zeta=e^{\frac{r}{7}\cdot 2\pi\sqrt{-1}}$ and $1\leq r\leq 6$. That
is, $p=7, m=1, (e_0, \cdots, e_5)= (0,1,-1,3,2,-3)$. Going back to
(\ref{eqn:forclassify}), we easily work out all solutions for
$\alpha_i$'s:
\begin{eqnarray*}
 (\alpha_0, \cdots, \alpha_5)&=&(2,1,0,0,0,0), (0,2,1,0,0,0), (0,0,2,1,0,0), (0,0,0,2,1,0),\\& & (0,0,0,0,2,1), (1,0,0,0,0,2), (1,0,1,0,1,0), (0,1,0,1,0,1).
\end{eqnarray*}
Thus the corresponding family
$$\bar B=\P\left(\Span\langle x_0^2x_1, x_1^2x_2, x_2^2x_3, x_3^2x_4, x_4^2x_5, x_5^2x_0, x_0x_2x_4, x_1x_3x_5
\rangle\right).$$ As before, to show that a general member of this
family is smooth, we only need to remark that $x_0^2x_1+
x_1^2x_2+x_2^2x_3+x_3^2x_4+x_4^2x_5+x_5^2x_0$ is smooth.
This accomplishes Family II in Theorem \ref{thm:classification}.\\

\noindent\textbf{Case III.} When $p=5$.\\ As before, by Lemma
\ref{lemma:cycle}, cycles of length 2,3,5 or 6 cannot occur. Thus
the only possible lengths of cycles are 1 and 4; and except the
trivial Family 0, there must be a 4-cycle:
$$2e_0+e_1=2e_1+e_2=2e_2+e_3=2e_3+e_0=j \mod 5^m.$$ As before, by (\ref{eqn:cycle}) we get
$$15e_0=5j \mod 5^m.$$ There thus exists $r\in \Z/5\Z$, such that $j=3e_0+r\cdot5^{m-1} \mod 5^m$, and
\begin{equation*}
 \begin{cases}
e_0=e_0;\\
e_1=e_0+r\cdot5^{m-1};\\
e_2=e_0-r\cdot5^{m-1};\\
e_3=e_0-2r\cdot5^{m-1};\\
 \end{cases}
\mod 5^m
\end{equation*}
Since $e_0\neq e_1$, $r$ cannot be 0. Since 2-cycle does not exists,
for $i=4, 5$, either $e_i$ takes the same value as $e_0,\cdots,
e_3$, or it is a 1-cycle, \ie $3e_i=j$. In any case, we can write
\begin{equation*}
 \begin{cases}
e_4=e_0+ar\cdot5^{m-1};\\
e_5=e_0+br\cdot5^{m-1},\\
 \end{cases}
\mod 5^m
\end{equation*}
where $a, b\in\Z/5\Z$. Taking into account of the first equation in
(\ref{eqn:forclassify}), we obtain
$$a+b=4 \mod 5.$$
As a result,
$$f=\left(\begin{array}{cccccc}
1&&&&&\\
           &\zeta &&&&\\
&&\zeta^{-1}&&&\\
&&&\zeta^{-2}&&\\
&&&&\zeta^{a}&\\
&&&&&\zeta^{4-a}          \end{array} \right)$$ where
$\zeta=e^{\frac{r}{5}\cdot 2\pi\sqrt{-1}}$ for $1\leq r\leq 4$ and
$a\in \Z/5\Z$.

That is, $p=5, m=1, (e_0, \cdots, e_5)= (0,1,-1,-2,a,4-a)$. Going
back to (\ref{eqn:forclassify}), we work out the solutions for
$\alpha_i$'s depending on the value of $a$:

\noindent\textbf{Subcase III (i).} When $a=0$.\\
$p=5, m=1, (e_0, \cdots, e_5)= (0,1,-1,-2,0,-1)$, and
$$f=\left(\begin{array}{cccccc}
1&&&&&\\
           &\zeta &&&&\\
&&\zeta^{-1}&&&\\
&&&\zeta^{-2}&&\\
&&&&1&\\
&&&&&\zeta^{-1}          \end{array} \right)$$ where
$\zeta=e^{\frac{r}{5}\cdot 2\pi\sqrt{-1}}$ for $1\leq r\leq 4$.
Solving $\alpha_i$'s from the equation (\ref{eqn:forclassify}):
\begin{eqnarray*}
 (\alpha_0, \cdots, \alpha_5)&=&(2,1,0,0,0,0), (0,2,1,0,0,0), (0,0,2,1,0,0), (1,0,0,2,0,0),\\& & (0,1,0,0,2,0), (0,0,0,2,1,0), (0,2,0,0,0,1), (0,0,0,1,0,2),\\ & & (1,1,0,0,1,0), (0,0,1,1,0,1).
\end{eqnarray*}
Thus the corresponding family
$$\bar B=\P\left(\Span\langle x_0^2x_1, x_1^2x_2, x_2^2x_3, x_3^2x_0, x_4^2x_1, x_3^2x_4, x_1^2x_5, x_5^2x_3, x_0x_1x_4, x_2x_3x_5 \rangle\right).$$
However, there is no smooth cubic fourfolds in this family: in fact
each member would have two singular points in the line
$(x_0=x_1=x_3=x_4=0)$.

\noindent\textbf{Subcase III (ii).} When $a=1$.\\
$p=5, m=1, (e_0, \cdots, e_5)= (0,1,-1,-2,1,-2)$, and
$\zeta=e^{\frac{r}{5}\cdot 2\pi\sqrt{-1}}$ for $1\leq r\leq 4$
$$f=\left(\begin{array}{cccccc}
1&&&&&\\
           &\zeta &&&&\\
&&\zeta^{-1}&&&\\
&&&\zeta^{-2}&&\\
&&&&\zeta&\\
&&&&&\zeta^{-2}          \end{array} \right)$$ By the transformation
$\zeta\mapsto\zeta^3$ (which amounts to let $r\mapsto 3r$), this $f$
is exactly the one in Subcase III(i).

\noindent\textbf{Subcase III (iii).} When $a=2$.\\
$p=5, m=1, (e_0, \cdots, e_5)= (0,1,-1,-2,2,2)$, and
$$f=\left(\begin{array}{cccccc}
1&&&&&\\
           &\zeta &&&&\\
&&\zeta^{-1}&&&\\
&&&\zeta^{-2}&&\\
&&&&\zeta^2&\\
&&&&&\zeta^{2}          \end{array} \right)$$ where
$\zeta=e^{\frac{r}{5}\cdot 2\pi\sqrt{-1}}$ for $1\leq r\leq 4$.
Solving $\alpha_i$'s from the equation (\ref{eqn:forclassify}):
\begin{eqnarray*}
 (\alpha_0, \cdots, \alpha_5)&=&(2,1,0,0,0,0), (0,2,1,0,0,0), (0,0,2,1,0,0), (1,0,0,2,0,0),\\& & (0,0,0,0,2,1), (0,0,0,0,1,2), (0,0,0,0,0,3), (0,0,0,0,3,0),\\& & (1,0,1,0,1,0), (1,0,1,0,0,1),  (0,1,0,1,1,0), (0,1,0,1,0,1).
\end{eqnarray*}
Thus the corresponding family
$$\bar B=\P\left(\Span\langle x_0^2x_1, x_1^2x_2, x_2^2x_3, x_3^2x_0, x_4^2x_5, x_5^2x_4, x_5^3, x_4^3, x_0x_2x_4, x_0x_2x_5,x_1x_3x_4, x_1x_3x_5
\rangle\right).$$ Moreover, a general cubic fourfold in this family
is smooth. Indeed, we give a particular smooth member:
$x_0^2x_1+x_1^2x_2+x_2^2x_3+x_3^2x_0+x_4^3+x_5^3$. This is Family
III in Theorem \ref{thm:classification}.

\noindent\textbf{Subcase III (iv).} When $a=3$.\\ By the symmetry
between $a$ and $b$, it is the same case as Subcase III(ii), hence
as Subcase III(i).

\noindent\textbf{Subcase III (v).} When $a=4$.\\ By the symmetry between $a$ and $b$, it is the same case as Subcase III(i).\\

\noindent\textbf{Case IV.} When $p=3$.\\ Still by
Lemma \ref{lemma:cycle}, we know that cycles of length 2,4 or 5
cannot occur. Thus the only possible lengths of cycles are 1, 3 and
6. We first claim that 6-cycle cannot exist. Suppose on the contrary
that the diagram is a 6-cycle:
$$2e_0+e_1=2e_1+e_2=2e_2+e_3=2e_3+e_4=2e_4+e_5=2e_5+e_0=j \mod 3^m,$$ then we have as in (\ref{eqn:cycle}) that
$63e_0=21j \mod 3^m.$ There thus exists $r\in \Z/3\Z$, such that
$j=3e_0+r\cdot3^{m-1} \mod 3^m$, and
\begin{equation*}
 \begin{cases}
e_0=e_0;\\
e_1=e_0+r\cdot3^{m-1};\\
e_2=e_0-r\cdot3^{m-1};\\
e_3=e_0;\\
e_4=e_0+r\cdot3^{m-1};\\
e_5=e_0-r\cdot3^{m-1}.\\
 \end{cases}
\mod 3^m
\end{equation*}
This contradicts to the assumption that $e_i$'s are distinct.
Therefore, there are only 1-cycles and 3-cycles. A 1-cycle means
$3e_i=j \mod 3^m$. On the other hand, a 3-cycle
$2e_0+e_1=2e_1+e_2=2e_2+e_0=j \mod 3^m$ would imply $9e_0=3j$. In
particular, $9e_0=9e_1=\cdots=9e_5=3j \mod 3^m$. Without loss of
generality, we can demand $e_0=0$. As a result, $f$ has the form
$f=\diag(1, \zeta^{a_1},\cdots, \zeta^{a_5})$ where
$\zeta=e^{\frac{2\pi\sqrt{-1}}{9}}$. In particular, $f$ is of order
3 or 9.

\noindent\textbf{Subcases IV (i).} If $f$ is of order 3. \\
Let $\omega:=e^{\frac{2\pi\sqrt{-1}}{3}}$. Then up to isomorphism,
$f$ is one of the following automorphisms:
\begin{itemize}
 \item $\diag(1,1,1,1,1,\omega)$: this case does not satisfy condition $(*)$.
 \item $\diag(1,1,1,1,1,\omega^2)$: this case does not satisfy condition $(*)$.
 \item $\diag(1,1,1,1,\omega,\omega^2)$: we find Family IV-(1) in Theorem
\ref{thm:classification}. We remark that its general member is
indeed smooth because in particular the Fermat cubic fourfold (which
is smooth) is contained in this family.
 \item $\diag(1,1,1,1,\omega,\omega)$: this case does not satisfy condition $(*)$.
 \item $\diag(1,1,1,1,\omega^2,\omega^2)$: this case does not satisfy condition $(*)$.
 \item $\diag(1,1,1,\omega,\omega,\omega^2)$: Here we find $\bar B$
 has a basis:
 $$x_5\cdot \text{ degree 2 monomials on } x_0, x_1 \text{ and } x_2;
 x_4x_5^2, x_3x_5^2, x_0x_3x_4, x_1x_3x_4, x_2x_3x_4, x_0x_3^2,
 x_1x_3^2, x_2x_3^2, x_0x_4^2, x_1x_4^2, x_2x_4^2,$$ However, any
 cubic fourfold in this family is singular along a conic curve in the
 projective plane $(x_3=x_4=x_5=0)$.
 \item $\diag(1,1,1,\omega^2,\omega^2,\omega)$: this is as in the previous case, with $\omega$ be replaced by $\omega^2$.
 \item $\diag(1,1,1,\omega,\omega,\omega)$:  By solving (\ref{eqn:forclassify}), we
 find a the following basis for $\bar B$:
 $$\text{degree 3 monomials on }x_0, x_1 \text{ and }x_2; \text{degree 3 monomials on }x_3, x_4 \text{ and
 }x_5.$$ As the Fermat cubic fourfold is in this family, the general
 member is also smooth. This is Family IV-(2) in Theorem \ref{thm:classification}.
 \item $\diag(1,1,\omega,\omega,\omega^2,\omega^2)$: The basis of $\bar B$ is
 $$\text{degree 3 monomials on }x_0\text{ and }x_1;\text{degree 3 monomials on }x_2\text{ and }x_3;\text{degree 3 monomials on }x_4\text{ and }x_5;$$
 $$x_0x_2x_4, x_0x_2x_5, x_0x_3x_4, x_0x_3x_5, x_1x_2x_4, x_1x_2x_5, x_1x_3x_4,
 x_1x_3x_5.$$ Because $\bar B$ contains the Fermat cubic fourfold,
 its general member is smooth. This is Family IV-(3) in Theorem \ref{thm:classification}.
\end{itemize}

\noindent\textbf{Subcase IV (ii).} If the diagram consists of two
3-cycles (thus $e_i$'s are distinct, so $m\geq 2$):
\begin{equation*}
 \begin{cases}
  2e_0+e_1=2e_1+e_2=2e_2+e_0=j\\
  2e_3+e_4=2e_4+e_5=2e_5+e_3=j\\
 \end{cases}
\mod 3^m,
\end{equation*}
From which we have $3j=9e_0=9e_3\mod 3^m$. Hence there exists $t=\pm
1$ such that $j=3e_0+t\cdot3^{m-1}$, and
\begin{equation*}
 \begin{cases}
e_0=e_0;\\
e_1=e_0+t\cdot3^{m-1};\\
e_2=e_0-t\cdot3^{m-1};\\
e_3=e_0+r\cdot3^{m-2};\\
e_4=e_0+t\cdot3^{m-1}-2r\cdot3^{m-2};\\
e_5=e_0-t\cdot3^{m-1}+4r\cdot3^{m-2}.\\
 \end{cases}
\mod 3^m
\end{equation*}
where  $r\in \Z/9\Z$. Note that $r\neq 0,3,6 \mod 9$, since
otherwise $e_i$'s cannot be distinct. By the first equation in
(\ref{eqn:forclassify}), $$t=-r \mod 3.$$ Putting this back into the
previous system of equations, we obtain:

$p=3, m=2, n=9, (e_0, \cdots, e_5)= (0,-3,3,1,4,-2), j=-3 \mod 9$,
and $$f=\left(\begin{array}{cccccc}
1&&&&&\\
           &\zeta^{-3} &&&&\\
&&\zeta^{3}&&&\\
&&&\zeta&&\\
&&&&\zeta^4&\\
&&&&&\zeta^{-2}          \end{array} \right)$$ where
$\zeta=e^{\frac{r}{9}\cdot 2\pi\sqrt{-1}}$ for
$r\in\{1,2,4,5,7,8\}$. Solving $\alpha_i$'s from the equation
(\ref{eqn:forclassify}):
\begin{eqnarray*}
 (\alpha_0, \cdots, \alpha_5)&=&(2,1,0,0,0,0), (0,2,1,0,0,0), (1,0,2,0,0,0), (0,0,0,2,1,0),\\& & (0,0,0,0,2,1), (0,0,0,1,0,2).
\end{eqnarray*}
Thus the corresponding family
$$\bar B=\P\left(\Span\langle x_0^2x_1, x_1^2x_2, x_2^2x_0, x_3^2x_4, x_4^2x_5, x_5^2x_3\rangle\right).$$
Clearly, the cubic
$x_0^2x_1+x_1^2x_2+x_2^2x_0+x_3^2x_4+x_4^2x_5+x_5^2x_3$ is smooth,
hence so is the general cubic fourfold in this family. This is
Family IV-(4) in Theorem \ref{thm:classification}.

\noindent\textbf{Subcase IV (iii).} If the diagram contains only one
3-cycle: $$2e_0+e_1=2e_1+e_2=2e_2+e_0=j \mod 9.$$ As before, we can
assume $e_0=0$, then $e_1=j$, $e_2=-j$ and $3j=0 \mod 9$. In
particular, $3|j$. Since $j\neq 0 \mod 9$ (otherwise $e_0=e_1=e_2$
is a contradiction), $j=\pm3$. For $i=3,4,5$, $e_i$ either takes value in $\{e_0, e_1, e_2\}$, or $3e_i=j$.\\
If $j=3$, then $f$ has the form
$$f=\diag(1, \zeta^3, \zeta^{-3}, \zeta^a, \zeta^b, \zeta^c),$$
where $a,b,c\in\{0,3,6,1,4,7\}$. By the first equation in
(\ref{eqn:forclassify}), $$a+b+c=6 \mod 9.$$ Thus either $a,b,c\in
\{0,3,6\}$, or $a,b,c\in \{1,4,7\}$. While the former will make $f$
of order 3, which has been treated in Subcases IV(i). Therefore
$a,b,c\in \{1,4,7\}$ and $a+b+c=6$. There are only three
possibilities (up to permutations of $a,b,c$): $(a,b,c)=(1,1,4)$ or
$(4,4,7)$ or $(7,7,1)$. However these three corresponds to the
following same automorphism
$$f=\left(\begin{array}{cccccc}
1&&&&&\\
           &\zeta^{3} &&&&\\
&&\zeta^{-3}&&&\\
&&&\zeta&&\\
&&&&\zeta&\\
&&&&&\zeta^{4}          \end{array} \right)$$ Back to
(\ref{eqn:forclassify}), we solve the corresponding $\alpha_i$'s to
get the following basis for $\bar B$:
$$\bar B=\P\left(\Span\langle x_0^2x_1, x_1^2x_2, x_2^2x_0, x_3^2x_4, x_4^2x_3, x_3^3, x_4^3,
x_5^3\rangle\right).$$ As we have a smooth member $x_0^2x_1+
x_1^2x_2+x_2^2x_0+x_3^3+x_4^3+x_5^3$ is this family, the general one
is also smooth. This is Family IV-(5) in Theorem \ref{thm:classification}.\\
If $j=-3$, it reduces to the $j=3$ case by replace $\zeta$ by
$\zeta^{-1}$, thus already included in Family IV-(5) of the theorem.

\noindent\textbf{Subcase IV (iv).} If the diagram has only 1-cycles,
\ie for any $0\leq i\leq 5$,
$$3e_i=j \mod 9.$$ In particular, $3|j$. First of all, $j\neq 0$,
otherwise, $f$ is of order 3, which is treated in Subcases IV(i).\\
If $j=3$. Then $e_i\in \{1,4,7\}$ for any $i$. Taking into account
the first equation of (\ref{eqn:forclassify}), we find all the
solutions for $(e_0,\cdots,e_5)$, up to permutations:
\begin{eqnarray*}
  (e_0,\cdots,e_5)&=& (1,1,1,4,4,4), (1,1,1,7,7,7), (4,4,4,7,7,7)\\
  &&(1,1,1,1,4,7), (4,4,4,4,1,7), (7,7,7,7,1,4)\\
  &&(1,1,4,4,7,7)
\end{eqnarray*}
where the automorphisms in the first line is equal to
$\diag(1,1,1,\omega,\omega,\omega)$, which has been done in Family
IV-(2); the automorphisms in the second line is equal to
$\diag(1,1,1,1,\omega,\omega^2)$, which has been done in Family
IV-(1); the last automorphism is equal to
$\diag(1,1,\omega,\omega,\omega^2,\omega^2)$, which has been done in
Family IV-(3) in Theorem \ref{thm:classification}.\\

\noindent\textbf{Case V.} When $p=2$.\\
By Lemma \ref{lemma:cycle}, we find that the associated
diagram has only 1-cycles. The new phenomenon is that the coming-in
degree in this case is not necessarily 1. Firstly, we claim that the
order of $f$ divides 32. Indeed, for any 1-cycle, say, $3e_0=j \mod
2^m$, then $e_0$ is in fact well-determined $\mod 2^m$. Without loss
of generality, we can assume that all the 1-cycles are 0, \ie $j=0$.
As a result, a vertex pointing to a 1-cycle is divisible by
$2^{m-1}$, and a vertex pointing to a vertex pointing to a 1-cycle
is divisible by $2^{m-2}$, \etc. As there is no cycle of length
$\geq 2$, every vertex, after at most 5 steps, arrives at some
1-cycle vertex. Therefore, every vertex is divisible by $2^{m-5}$,
hence we can reduce everything modulo 32: namely $n=32$ and $e_i\in
\Z/32\Z$.

\begin{figure}[h]
\begin{center}
 \includegraphics[scale=3, width=13 cm]{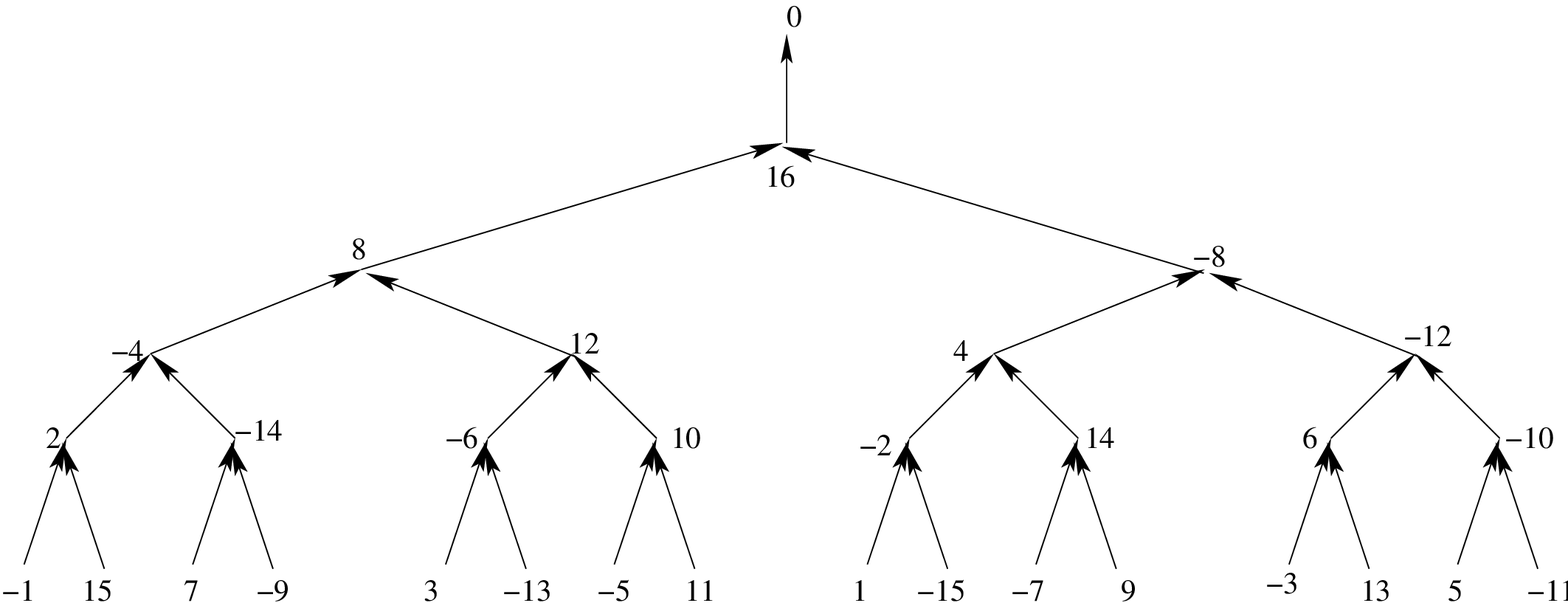}
\end{center}
\end{figure}

Let us put $\Z/32\Z$ into the above complete binary tree, then
clearly, our associated diagram is a sub-diagram of this tree,
satisfying two properties:
\begin{itemize}
 \item If a vertex belongs to the diagram then so do its ancestors;
 \item The sum of vertices (multiplicities counted) is zero modulo 32.
\end{itemize}

It is immediate that the leaves (vertices on the bottom sixth level)
cannot appear in the diagram: since by the parity of their sum, if
there are leaves, there are at least two. But we already have five
ancestors to include, while we have only six places in total. Next,
we remark that the vertices in the fifth level cannot belong to the
diagram neither: since the sum is divisible by 4, there are at least
two vertices from the fifth level if there is any, and they should
have the same father (otherwise we need to include at least five
ancestors, and it will be out of place). Therefore we only have four
possibilities, and it is straightforward to check that none of them
has sum zero as demanded.

As a result, we have a further reduction: since only the first four
levels can appear, the order of $f$ always divides 8. We can assume
now $n=8$ and $e_i\in \Z/8\Z$. Similarly, we put $\Z/8\Z$ into the
following complete binary tree:
\begin{figure}[h]
\begin{center}
 \includegraphics[scale=2, width=5 cm]{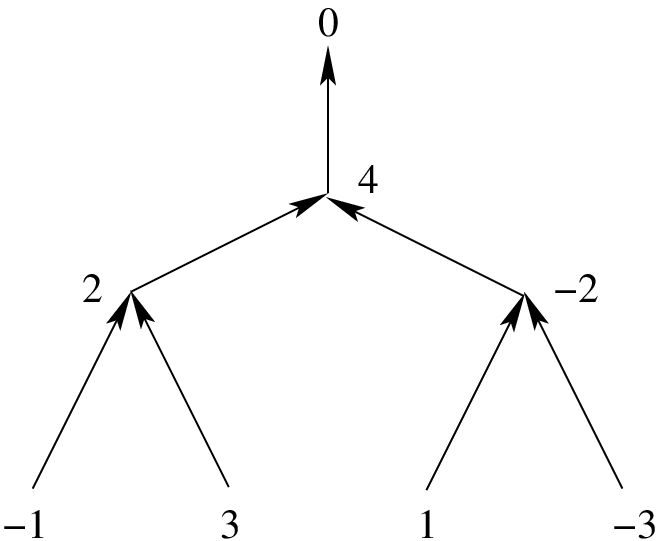}
\end{center}
\end{figure}

Then our diagram is a sub-diagram of this tree which is as before
`ancestor-closed' and of multiplicities counted sum zero modulo 8.
We easily work out all the possibilities as follows. It is worthy to
point out that when $p=2$, for a given automorphism $f$, there may
be two possible values of $j$, which would correspond to two
different families of cubic fourfolds.
\begin{itemize}
 \item $(e_0,\cdots, e_5)=(0,0,0,0,4,4)$ or $(0,0,4,4,4,4)$. In this case,
  $f$ is the involution $\diag(1,1,1,1,-1,-1)$ and we reduce to $n=p=2$ with $(e_0,\cdots,
  e_5)=(0,0,0,0,1,1)$. It splits into two cases depending on the parity of $j$:\\
  If $j$ is even, then the equation for $\alpha_i$'s becomes $\alpha_4+\alpha_5=0 \mod
  2$. This is Family V-(1) in Theorem \ref{thm:classification}, whose
  generic smoothness is easy to verify: $x_4^2x_0+x_5^2x_1+x_4x_5x_2+x_0^3+x_1^3+x_2^3+x_3^3$ is smooth. This family of
  cubic fourfolds has been studied in \cite{MR2967237}.\\
  If $j$ is odd, then a basis for $\bar B$ is given by
  $$x_5\cdot\text{degree 2 monomials on} x_0,\cdots,x_3; x_4\cdot\text{degree 2 monomials on} x_0,\cdots,x_3;x_4^3, x_5^3, x_4^2x_5, x_5^2x_4.$$
   However, any cubic fourfold
  in this family is singular along the intersection of two quadrics
  in the projective 3-planes $(x_4=x_5=0)$.
 \item $(e_0,\cdots, e_5)=(0,0,0,4,2,2)$ or $(0,4,4,4,-2,-2)$.
 They both correspond to the following automorphism of order $n=4$:
 $$f=\diag(1,1,1,-1,i,i).$$ We thus reduce to $n=4$, and $(e_0,\cdots,
 e_5)=(0,0,0,2,1,1)$. Therefore $2j=0 \mod 4$, and thus $j=0, 2\mod
 4$. It also splits into two different cases depending on the value of
 $j$.\\
 If $j=0\mod 4$, the equation (\ref{eqn:forclassify}) for $\alpha_i$'s becomes
 $2\alpha_3+\alpha_4+\alpha_5=0 \mod 4$. We easily obtain a basis
 for $\bar B$:
 $$\text{degree 3 monomials on }x_0,x_1, x_2; x_0x_3^2, x_1x_3^2, x_2x_3^2, x_3x_4x_5, x_3x_4^2,
 x_3x_5^2.$$ Unfortunately, any cubic fourfold in this family is
 singular on two points on the line defined by
 $(x_0=x_1=x_2=x_3=0)$.\\
 If $j=2\mod 4$, the equation is $2\alpha_3+\alpha_4+\alpha_5=2 \mod
 4$. The following consists of a basis for $\bar B$:
 $$x_3\cdot\text{degree 2 monomials on }x_0,x_1, x_2; x_3^3, x_0x_4^2,x_1x_4^2, x_2x_4^2, x_0x_5^2,x_1x_5^2, x_2x_5^2, x_0x_4x_5,
 x_1x_4x_5,x_2x_4x_5.$$ The general member is singular along a conic
 in the projective plane $(x_3=x_4=x_5=0)$.

 \item $(e_0,\cdots, e_5)=(0,0,0,4,-2,-2)$ or $(0,4,4,4,2,2)$. They both correspond to the same $f$, which becomes the automorphism of the previous case if we replace $i$ by $-i$.
 \item $(e_0,\cdots, e_5)=(0,4,2,2,2,-2)$ or $(0,4,2,-2,-2,-2)$. They both correspond to the following automorphism of order $n=4$:
 $$f=\diag(1,1,1,-1,\sqrt{-1},-\sqrt{-1}).$$ We thus reduce to $n=4$
 and $(e_0,\cdots, e_5)=(0,0,0,2,1,-1)$. Hence $2j=2\mod 4$ gives
 two cases $j=1,-1\mod 4$.\\
 If $j=1 \mod 4$, the equation for $\alpha_i$'s becomes $2\alpha_3+\alpha_4-\alpha_5=1 \mod
 4$. The basis for $\bar B$ is:
 $$x_4\cdot\text{degree 2 monomials on }x_0,x_1,x_2; x_3^2x_4, x_0x_3x_5, x_1x_3x_5,x_2x_3x_5,x_5^3,
 x_4^2x_5.$$ Each cubic fourfold in this family is singular along a
 conic curve in the plane defined by $(x_3=x_4=x_5=0)$.\\
 If $j=-1 \mod 4$, the equation for $\alpha_i$'s becomes $2\alpha_3+\alpha_4-\alpha_5=-1 \mod
 4$, all the solutions are exactly the ones when $j=1$ with $x_4$
 and $x_5$ interchanged.

 \item $(e_0,\cdots, e_5)=(0,0,4,4,2,-2)$. It is the following automorphism of order $n=4$:
 $$f=\diag(1,1,-1,-1,\sqrt{-1},-\sqrt{-1}).$$ We then reduce to $n=4$ and $(e_0,\cdots,
 e_5)=(0,0,2,2,1,-1)$. Therefore $2j=0\mod 4$, hence we have two
 cases.\\
 If $j=0\mod 4$, the equation for $\alpha_i$'s is $2\alpha_2+2\alpha_3+\alpha_4-\alpha_5=0\mod
 4$.It corresponds Family V-(2)(a) in Theorem
 \ref{thm:classification}. It is easy to find a smooth member, for
 example, $x_0^3+x_1^3+x_2^2x_0+x_3^2x_1+x_4^2x_2+x_5^2x_3$.\\
 If $j=2\mod 4$, the equation for $\alpha_i$'s is $2\alpha_2+2\alpha_3+\alpha_4-\alpha_5=2\mod
 4$, we have Family V-(2)(b) in Theorem
 \ref{thm:classification}. We give an example of smooth cubic
 fourfold in this family:
 $x_2^3+x_3^3+x_2x_0^2+x_3x_1^2+x_0x_4^2+x_1x_5^2$.

 \item $(e_0,\cdots, e_5)=(0,4,4,2,-1,-1)$ or $(0,4,4,2,3,3)$ or $(0,4,4,-2,1,1)$ or $(0,4,4,-2,-3,-3)$.
 Although they are \emph{different} automorphisms of order $n=8$,
 they correspond to four possible choices of the primitive eighth root of unity $\zeta$ in the automorphism:
 $$f=\diag(1,-1,-1,\zeta^{-2},\zeta,\zeta),$$
 where $\zeta=e^{\frac{r}{8}\cdot 2\pi\sqrt{-1}}$ for $r=\pm1,\pm3$.
 For each choice of $r$, one of the two possible values of $j$ does not satisfies the last condition $(*)$ in (\ref{eqn:forclassify}).
 The remaining one gives the equation $$4\alpha_1+4\alpha_2-2\alpha_3+\alpha_4+\alpha_5=0 \mod
 8.$$ The solutions form a basis for $\bar B$:
 $$\bar B=\P\left(\Span\langle x_0^3, x_0x_1^2,x_0x_2^2,x_0x_1x_2,
 x_2x_3^2,x_1x_3^2,x_3x_4^2,x_3x_5^2,x_3x_4x_5\rangle\right).$$
 However any cubic fourfold in this family is singular on two points
 on the line defined by $(x_0=x_1=x_2=x_3=0)$.

\item $(e_0,\cdots, e_5)=(0,0,4,2,-1,3)$ or $(0,0,4,-2,1,-3)$.
They are \emph{different} automorphisms of order $n=8$. In fact the
four possible choices of the primitive eighth root of unity $\zeta$
collapse into two cases. The automorphism:
 $$f=\diag(1,1,-1,\zeta^{2},\zeta^{-1},\zeta^3),$$
 where $\zeta=e^{\frac{r}{8}\cdot 2\pi\sqrt{-1}}$ for $r=\pm1$ (here $r=\pm3$ will give the same two automorphisms).
 In this case, one of the two possible values of $j$ does not satisfies the last condition $(*)$ in (\ref{eqn:forclassify}).
 The remaining one corresponds to the equation $$4\alpha_2+2\alpha_3-\alpha_4+3\alpha_5=0\mod
 8.$$ We easily resolve it to obtain $$\bar B=\P\left(\Span\langle \text{degree 3 monomials on }x_0 \text{ and }x_1, x_0x_2^2,x_1x_2^2,x_2x_3^2, x_3x_4^2,
 x_3x_5^2\rangle\right).$$ But each member in this family is
 singular at least on two points of the line defined by $(x_0=x_1=x_2=x_3=0)$.

\item $(e_0,\cdots, e_5)=(0,4,2,-2,1,3)$ or $(0,4,2,-2,-1,-3)$.
As in the previous case, although they are \emph{different}
automorphisms of order $n=8$, each corresponds to two possible
choices of the primitive eighth root of unity $\zeta$. The
automorphism is $$f=\diag(1,-1,\zeta^{2},\zeta^{-2},\zeta,
\zeta^3),$$ where $\zeta=e^{\frac{r}{8}\cdot 2\pi\sqrt{-1}}$ for
$r=\pm1$ (here $r=\pm3$ will give the same two automorphisms).  In
this case, one of the two possible values of $j$ does not satisfies
the last condition $(*)$ in (\ref{eqn:forclassify}). The remaining
one is Family V-(3) in Theorem \ref{thm:classification}, where the
smoothness of the general member is affirmed by the example:
$x_0^3+x_0x_1^2+x_1x_2^2+x_1x_3^2+x_0x_2x_3+x_3x_4^2+x_2x_5^2$.
Since the verification of the smoothness of this example is a little
bit involved, we give the details here. Let
$T=x_0^3+x_0x_1^2+x_1x_2^2+x_1x_3^2+x_0x_2x_3+x_3x_4^2+x_2x_5^2$. To
find the singular locus of $(T=0)$, we need to solve the system of
equations $$\frac{\partial T}{\partial x_0}=\frac{\partial
T}{\partial x_1}=\cdots=\frac{\partial T}{\partial x_5}=0,$$ that
is:
\begin{equation*}
  \begin{cases}
    3x_0^2+x_1^2+x_2x_3=0;\\
    2x_0x_1+x_2^2+x_3^2=0;\\
    2x_1x_2+x_0x_3+x_5^2=0;\\
    2x_3x_1+x_0x_2+x_4^2=0;\\
    x_3x_4=0;\\
    x_2x_5=0.
  \end{cases}
\end{equation*}
Thanks to the last two equations, we have four cases $(x_2=x_3=0),
(x_2=x_4=0),(x_5=x_3=0), (x_5=x_4=0)$. In the first three cases, it
is easy to deduce that every variable is zero. In the last case, the
system of equations simplifies to:
\begin{equation*}
  \begin{cases}
    3x_0^2+x_1^2+x_2x_3=0;\\
    2x_0x_1+x_2^2+x_3^2=0;\\
    2x_1x_2+x_0x_3=0;\\
    2x_3x_1+x_0x_2=0.
  \end{cases}
\end{equation*}
It is still easy to deduce that every variable is non-zero. And then
from the last two equations, we find $x_0=\pm 2x_1$ and $x_2=\mp
x_3$. Putting these into the first two equations, we get
contradictions. As a consequence, $(T=0)$ is smooth.
\end{itemize}

The classification is complete and the result is summarized in
Theorem \ref{thm:classification}.

\begin{rmks}\label{rmk:final}
We have some explanations to make concerning the usage of our list.
\begin{itemize}
 \item In the fifth column of the table in Theorem \ref{thm:classification},
 we give a basis for the \emph{compactified} parameter space $\bar B$,
 which contains of course singular members. To pick out the smooth ones (\ie to determine the \emph{non-empty} open dense subset $B$),
 we have to apply usual method of Jacobian criterion.
 \item Strictly speaking, the moduli space of cubic fourfolds is the geometric quotient $$M:=\P\left(H^0(\P^5,\sO(3))\right)//\PGL_6,$$
 and each $\bar B$ we have given in the theorem is a sub-projective space of $\P\left(H^0(\P^5,\sO(3))\right)$, whose image in $M$ is (a component of) the `moduli space' of cubic fourfolds admitting a `symplectic' automorphism of certain primary order.
 \item For an automorphism $f$ of a given order $n$, say
 $n=2^{r_2}3^{r_3}5^{r_5}7^{r_7}11^{r_{11}}$, where $r_2=$ 0, 1, 2 or 3; $r_3=$ 0,1 or 2 and $r_5,r_7,r_{11}=$ 0 or
 1. Then $f=f_2f_3f_5f_7f_{11}$ where $f_p$ is an automorphisms of order $p^{r_p}
 $ commuting with each other. Thus they can be diagonalised
 simultaneously. Therefore to classify automorphisms of a given
 order, it suffices to  intersect the corresponding families $\bar B$'s in the list, after \emph{independent} scaling and permutation of coordinates, inside the complete linear system $\P\left(H^0(\P^5,\sO(3))\right)$.
 Of course it may end up with an empty family or a family consisting of only singular members.
\end{itemize}
\end{rmks}

\begin{expl}
  We investigate the example of Fermat cubic fourfold
  $X=(x_0^3+x_1^3+x_2^3+x_3^3+x_4^3+x_5^3=0)$. We know that (\cf \cite{MR966448},
  \cite{MR1894739}) its automorphism group is $\Aut(X)=\left(\Z/3\Z\right)^5\rtimes
  \mathfrak{S}_6$, which is generated by multiplications by third
  roots of unity on coordinates and permutations of coordinates. Using Griffiths' residue description of Hodge
  structure as in the proof of Lemma \ref{lemma:symp}, we find that
  $$\Aut^{pol, symp}(F(X))=\left\{f\in\Aut(X)~|~f^*|_{H^{3,1}(X)}=\id\right\}=\left(\Z/3\Z\right)^4\rtimes
  \mathfrak{A}_6,$$ where each element has the form:
  $$[x_0, x_1, x_2, x_3, x_4, x_5]\mapsto
  [x_{\sigma(0)},\omega^{i_1}x_{\sigma(1)},\cdots,\omega^{i_5}x_{\sigma(5)}],$$
  where $\omega=e^{\frac{2\pi\sqrt{-1}}{3}}$, $i_1,\cdots, i_5=$ 0, 1 or 2 with sum $i_1+\cdots+i_5$
  divisible by 3, and $\sigma\in \mathfrak{A}_6$ is a permutation of
  $\{0,1,\cdots,5\}$ with even sign.

    Then $X$ is
  \begin{itemize}
    \item not in Family I, II, IV-(4), IV-(5) or V-(3) simply because
    $X$ does not admit automorphisms of order 11, 7, 9 or 8;
    \item in Family III, since $[x_0, x_1, x_2, x_3, x_4, x_5]\mapsto [x_1, x_2, x_3, x_4,x_0,
    x_5]$ is an order 5 automorphism which induces a symplectic automorphism on its Fano variety of lines. The eigenvalues of the
    corresponding permutation matrix is $1,1,\zeta,
    \zeta^2,\zeta^3,\zeta^4$, thus it is exactly the automorphism in
    the list (up to a linear automorphism of $\P^5$).
    \item in IV-(1), IV-(2), IV-(3) obviously;
    \item in V-(1), V-(2)(a) and V-(2)(b), because $[x_0, x_1, x_2, x_3, x_4, x_5]\mapsto [x_1, x_2,
    x_3,x_0, x_5,x_4]$ is an order 4 automorphism which induces a
    symplectic automorphism of its Fano variety. The eigenvalues are
    $1,1,-1,-1,\sqrt{-1},-\sqrt{-1}$, therefore the automorphism is
    the one given in V-(2) (up to a linear automorphism of $\P^5$).
  \end{itemize}

\end{expl}

\section{Fixed loci}\label{sect:fixed}
We calculate the fixed loci of a generic member for each example in the list of Theorem
\ref{thm:classification}. Firstly, we make several general remarks
concerning the fixed loci:
\begin{itemize}
  \item For a smooth variety,  the fixed locus of any automorphism of finite
  order is a (not necessarily connected) smooth subvariety. For a
  proof, \cf\cite[Lemma 4.1]{MR0263834}.
  \item If furthermore the variety is symplectic and the finite-order automorphism
  preserves the symplectic form, then the components of the fixed
  locus are \emph{symplectic} subvarieties. Indeed, for a given fixed point, the automorphism
  acts also on the tangent space at this fixed point, preserving the symplectic form,
  where the tangent space of the component of the fixed locus
  passing through this point is exactly the fixed subspace.
  However, since the fixed subspace is orthogonal to the other
  eigenspaces with respect to the symplectic form, it must be a
  symplectic subspace. In consequence, the fixed locus is a (smooth)
  symplectic subvariety.
\end{itemize}
Therefore, in the case of this paper, the fixed loci must be
disjoint unions of (isolated) points, K3 surfaces and abelian
surfaces, and we will see that all three types do occur in the
list in Theorem \ref{thm:classification}.

We now turn to the calculation of the fixed loci of the examples in
our classification. For a cubic fourfold $X$ with an action $f$, we
denote $\hat f$ the induced action on $F(X)$. Then the fixed points
of $\hat f$ in $F(X)$ are the lines contained in $X$ which are
preserved by $f$. Since any automorphism of $\P^1$ admits two (not
necessarily distinct) fixed points, it suffices to check for each
line joining two fixed points of $f$ in $X$ whether it is contained
in $X$. In the following, we choose some typical examples in our
list to give the argument in detail, while the complete result is
presented in the last column of the table in Theorem
\ref{thm:classification}.

Denote $P_0:=[1,0,0,0,0,0]$, $P_1:=[0,1,0,0,0,0],\cdots,$
$P_5:=[0,0,0,0,0,1]$. We have explicit description of the fixed
loci:

For \textbf{Family I}, a cubic fourfold in this family has equation
of the following form
$$a_0x_0^2x_1+a_1x_1^2x_2+a_2x_2^2x_3+a_3x_3^2x_4+a_4x_4^2x_0+a_5x_5^3.$$
The fixed points of $f$ in $X$ are $P_0, P_1, P_2, P_3, P_4$. We
check the 10 possible lines joining two of them and find that only
the following 5 are contained in $X$: $\bar{P_0P_2}$,
$\bar{P_0P_3}$, $\bar{P_1P_3}$, $\bar{P_1P_4}$, $\bar{P_2P_4}$,
where $\bar{PQ}$ means the line joining two points $P$ and $Q$.
Similarly, the same argument applies in the following
families and gives:\\
\textbf{Family II}: the fixed points in $F(X)$ are given by the following nine lines: $\bar{P_0P_2}$, $\bar{P_0P_3}$, $\bar{P_0P_4}$, $\bar{P_1P_3}$, $\bar{P_1P_4}$, $\bar{P_1P_5}$, $\bar{P_2P_4}$, $\bar{P_2P_5}$, $\bar{P_3P_5}$.\\
\textbf{Family IV-(4)}: the fixed points of $F(X)$ correspond to the
following nine lines: $\bar{P_0P_3}$, $\bar{P_0P_4}$,
$\bar{P_0P_5}$, $\bar{P_1P_3}$, $\bar{P_1P_4}$, $\bar{P_1P_5}$,
$\bar{P_2P_3}$, $\bar{P_2P_4}$, $\bar{P_2P_5}$.\\
\textbf{Family V-(3)}: the fixed points are given by the six lines:
$\bar{P_1P_4}$, $\bar{P_1P_5}$, $\bar{P_2P_3}$, $\bar{P_2P_4}$,
$\bar{P_3P_5}$, $\bar{P_4P_5}$.

For \textbf{Family III}, the equation has the following form
$$C(x_4, x_5)+R(x_0,\cdots, x_5),$$ where $C$ is a homogeneous
polynomial of degree 3, and $R$ is a polynomial with the degrees of
$x_4$ and $x_5$ at most 1. The fixed points of $f$ in $X$ are $P_0,
P_1, P_2, P_3$ and the line $\bar{P_4P_5}$. On one hand, among the
six possible lines joining $P_0, P_1, P_2, P_3$, only $\bar{P_0P_2}$
and $\bar{P_1P_3}$ are contained in $X$; on the other hand, for
$0\leq i\leq 4$, the line $\bar{P_i,[0,0,0,0,\lambda,\mu]}$ is
contained in $X$ if and only if $[\lambda, \mu]$ satisfies the cubic
equation $C$, and therefore we have three for each $i$. Altogether,
the fixed locus in $F(X)$ consists of $2+4\times3=14$ lines. Similar
arguments gives the results of the following:\\
\textbf{Family IV-(3)}: the equation has the form
$C_1(x_0,x_1)+C_2(x_2,x_3)+C_3(x_4,x_5)+R$, where $C_i$ are of
degree 3 while each term of $R$ is square-free. Then the fixed locus
of in $F(X)$ corresponds to the 27 lines $\bar{Q_{ik}Q_{jl}}$ for
$0\leq i<j\leq 3$ and $k,l=1,2,3$, where $Q_{i1}, Q_{i2}, Q_{i3}$
are the three points satisfying the equation $C_i$.\\
\textbf{Family IV-(5)}: the equation writes $C(x_3,
x_4)+a_0x_0^2x_1+a_1x_1^2x_2+a_2x_2^2x_0+a_5x_5^3$, where $C$ is of
degree 3. Let $Q_1, Q_2, Q_3$ be the three points on the line
$\bar{P_3P_4}$ satisfying $C$. Then the fixed locus in $F(X)$
correspond to the 9 lines: $\bar{P_iQ_j}$ for $i=0,1,2$ and
$j=1,2,3$.

For \textbf{Family IV-(1)}, let defining equation of the cubic
fourfold be $$C(x_0, \cdots, x_3)+ R$$ where $C$ is of degree 3 and
each term of $R$ contains $x_4$ or $x_5$. Clearly, the fixed locus
of $f$ in $X$ is $\P^3=(x_4=x_5=0)$. A line in this $\P^3$ is
contained in $X$ if and only if it satisfies $C$, namely, it is
contained in the cubic surface defined by $C$. It is well-known that
there are 27 such lines.

For \textbf{Family IV-(2)}, let cubic fourfold be defined by
$C_1(x_0, x_1, x_2)+C_2(x_3, x_4, x_5)$, where $C_1, C_2$ are of
degree 3. The fixed locus in $X$ is two disjoint planes:
$W_1=(x_0=x_1=x_2=0)$ and $W_2=(x_3=x_4=x_5=0)$. On one hand, inside
each plane it is impossible to have a line contained in $X$. On the
other hand, a line joining a point $Q_1\in W_1$ and a point $Q_2\in
W_2$ is contained is $X$ if and only if $Q_i$ satisfies the equation
$C_i$ for $i=1,2$, \ie $Q_i$ is in the elliptic curve $E_i$ defined
by $C_i$. Thus such lines are parameterized by $E_1\times E_2$,
which is an abelian surface.

\textbf{Family V-(1)} is done also in \cite{MR2967237}, but we
reproduce the argument for the sake of completeness. The equation
has the following form:
$$C(x_0,\cdots, x_3)+x_4^2L_1+x_5^2L_2+x_4x_5L_3,$$ where $C$ is of
degree 3, and $L_1, L_2, L_3$ are linear forms in $x_0,\cdots,x_3$.
The fixed points of $f$ in $X$ is clearly the disjoint union of
$\P^3=(x_4=x_5=0)$ and the line $\bar{P_4P_5}$. First of all, the
line $\bar{P_4P_5}$ is contained in $X$, giving a isolated fixed
point in $F(X)$; secondly a line of this $\P^3$ is contained in $X$
if and only if it satisfies $C$, \ie it is contained in the cubic
surface defined by $C$, and we thus obtain another 27 isolated fixed
points in $F(X)$; finally, the condition that a line joining a point
$Q_1\in \P^3$ and another point $Q_2\in \bar{P_4P_5}$ is contained
in $X$ is given by a double cover of the cubic surface $(C=0)$
ramified along the degree 6 curve $(C=L_3^2-L_1L_2=0)$, which is a
K3 surface. Altogether, the fixed locus of $\hat f$ in $F(X)$ is 28
points with a K3 surface.

For \textbf{Family V-(2)(a)}, the fixed point set of $f$ consists of
the disjoint union of $\bar{P_0P_1}$, $\bar{P_2P_3}$ and $P_4$,
$P_5$. The line $\bar{P_4P_5}$ is contained in $X$; there are three
points $Q_1, Q_2, Q_3\in \bar{P_0P_1}$ such that $\bar{Q_iP_j}$ is
contained in $X$, for $i=1,2,3$ and $j=4,5$; there are two points
$Q_4, Q_5\in \bar{P_2P_3}$ such that $\bar{Q_4P_4}$ and
$\bar{Q_5P_5}$ are contained in $X$; finally for each $Q_i\in
\bar{P_0P_1}$, $1\leq i\leq 3$, there exists two points on
$\bar{P_2P_3}$ such that the joining line is contained in $X$. Thus
$\hat f$ has altogether $1+3\times2+2+3\times2=15$ isolated fixed
points. Similarly, for \textbf{Family V-(2)(b)}, the fixed locus in
$F(X)$ also consists of 15 isolated points.

\bibliographystyle{amsplain}
\bibliography{biblio_fulie}

\end{document}